\documentclass[letterpaper, 10 pt, conference]{ieeeconf} 
\IEEEoverridecommandlockouts                          
\usepackage{graphics} 
\usepackage{epsfig} 
\usepackage{tikz,cite}
\usepackage{authblk}
\usepackage{amsmath} 
\usepackage{amssymb}  
\usepackage{dblfloatfix}
\graphicspath{{./Figures/}}
\newcommand{\R}{\mathbb{R}}
\newcommand{\cc}{\mathbf{c}}
\newcommand{\x}{\mathbf{x}}
\newcommand{\dd}{\mathbf{d}}
\newcommand{\w}{\mathbf{w}}
\newcommand{\C}{\mathcal{C}}
\newtheorem{theorem}{\textbf{Theorem}}
\newtheorem{lemma}{Lemma}
\newtheorem{proposition}{Proposition}
\newtheorem{definition}{Definition}
\newtheorem{assumption}{Assumption}

\newtheorem{example}{Example}
\newtheorem{remark}{Remark}
\title{\LARGE \bf
Effect of Bonus Payments in Cost Sharing Mechanism Design for Renewable Energy Aggregation
}
\author{
	Farshad Harirchi, Tyrone Vincent and Dejun Yang
	\thanks{Farshad Harirchi is with the EECS Department, University of Michigan, Ann Arbor, MI, 48109. {\tt{harirchi@umich.edu}}}
	\thanks{Tyrone Vincent and Dejun Yang are with the EECS Department, Colorado School of Mines, Golden, CO, 80401. {\tt{\{tvincent, djyang\}@mines.edu}}}
}

\begin{document}
\maketitle
\thispagestyle{empty}
\pagestyle{empty}

\begin{abstract}
	The participation of renewable energy sources in energy markets is challenging, mainly because of the uncertainty associated with the renewables. Aggregation of renewable energy suppliers is shown to be very effective in decreasing this uncertainty. In the present paper, we propose a cost sharing mechanism that entices the suppliers of wind, solar and other renewable resources to form or join an aggregate. In particular, we consider the effect of a bonus for surplus in supply, which is neglected in previous work. We introduce a specific proportional cost sharing mechanism, which satisfies the desired properties of such mechanisms that are introduced in the literature, e.g., budget balancedness, ex-post individual rationality and fairness. In addition, we show that the proposed mechanism results in a stable market outcome. Finally, the results of the paper are illustrated by numerical examples.
	
\end{abstract}
\section{introduction}
Renewable energy resources are being widely employed because of their cleanness and low marginal cost. Large penetration of renewable energy is desirable in the future smart grids, because of these advantages \cite{Bitarselling2012}. The drawback of renewable energy sources is the uncertainty associated with them, which is caused by the uncertainty of wind or solar energy. There are subsidies for renewable energy suppliers in current markets \cite{Makarov2009, Bitarselling2012}, but such subsidies reduce the social welfare in the electricity markets and should not be applied in the future smart grids. Removing subsidies makes the future electricity markets more competitive for renewable energy producers, because they have to compete with other suppliers in the market in equal conditions. In such a market the uncertainty associated with renewable energy is a drawback. Various studies show that the renewable energy produced in different geographical areas often have negative correlation and aggregating them reduces the amount of uncertainty of these resources \cite{EnerNex2010,NERC2009 }. However, aggregating renewable energy suppliers is not possible without appropriate payment sharing mechanisms. Renewable aggregation mechanisms are considered within different system levels from small size residential units \cite{Babakmehr2016Design} to the large scale microgrids \cite{BitarNodal2012}. These mechanisms are often designed to result in a stable market outcome, and incentivize suppliers to join the aggregate by paying them more than what they earn outside the aggregate. The desired properties that a payment sharing mechanism should satisfy such as budget balancedness, ex-post individual rationality and fairness are investigated in the literature\cite{Nayyar2013, Harirchi2014, Lin2014}. 

\subsection{Literature Review}
Bitar \textit{et. al.} \cite{BitarOptimal2010, BitarBringing2012} addressed the problem of finding the optimal bid for the aggregate using convex optimization techniques. This work has been followed by \cite{Nayyar2013, BitarNodal2012, Harirchi2014, Lin2014, BaeyensCoalitional2013} with introducing the contract game and designing payment sharing mechanisms, which satisfy a list of properties and result in stable market outcome. Baeyens \textit{et. al.} \cite{BaeyensCoalitional2013} proposed a payment sharing mechanism based on the expected wind power production of each supplier, which depends on the probability distribution of the power of the aggregate. They also showed that the core of the coalitional contract game is not empty, and that there exists at least one set of payments that stabilizes the market. The payoffs for the suppliers are calculated by solving a convex optimization problem that is computationally demanding. 

Nayyar {\em et. al.} proposed a payment sharing mechanism that is based on each supplier's performance rather than the expected production \cite{Nayyar2013}. This approach provides results on the existence of Nash equilibria for the contract game defined by this payment sharing mechanism, however it does not optimize the payoff of the aggregate. In \cite{Harirchi2014}, a payment sharing mechanism is proposed in which the payments are based on the production of individual suppliers, and it optimizes the payoff for the aggregate. 

As an alternative to the above-mentioned mechanisms, here is considered a proportional cost sharing mechanism, which provides an intuitive measure of each supplier's portion in the total cost/bonus of the aggregate. 

Lin \textit{et. al.} \cite{Lin2014} investigated the proportional cost sharing mechanisms for renewable energy aggregation. They assumed that the aggregate can avoid the excess of supply. The effect of energy surplus for aggregate is neglected in their analysis. Even though tools such as demand response and storage can be employed to manage power excess in the grid, the effect of surplus cannot be completely neglected in designing cost sharing mechanisms. 

\subsection{Contributions and Outline}
In this paper, we propose a proportional cost sharing mechanism. This cost sharing mechanism is proved to satisfy the desired properties such as budget balancedness, ex-post individual rationality and fairness. The main contribution of the paper is that we consider surplus for the aggregate, and a bonus for surplus is allowed in our analysis.

In addition, we study the existence of Nash equilibria for the contract game defined by proposed cost sharing mechanism, which becomes more challenging when the effect of bonus is considered.

The structure of the paper is as follows: In Section \ref{Sec:model}, the forward electricity market model utilized in this work is described. The proportional cost sharing mechanisms and the effect of adding bonus for the surplus of energy are studied in Section \ref{Sec:proportional}. The contract game and the results on the existence of Nash equilibria are drawn in Section \ref{Sec:Nash}. Two illustrative examples are discussed in Section \ref{Sec:example}, and conclusions are made in Section \ref{Sec:conclusion}.

		
\section{Model} \label{Sec:model}
\textbf{Notation}: In this work, $\x\in \R^n$ represents a vector. We denote the $\max\{0,x\}$ with $[x]^+$, and the set of natural numbers up to $s$ is indicated by $\mathbb{N}_s$. The set of positive real numbers are denoted by $\mathbb{R}^+$

\subsection{Market Structure}
We consider a forward electricity market model that is very similar to the one used in \cite{Lin2014}. The market considered here is perfectly competitive, meaning that no supplier or consumer can affect the market clearing price. In such a market, consider an aggregate with a set $\mathcal{N} = \{1,2, \hdots , n\}$ of $n$ suppliers. Each supplier $i \in \mathcal{N}$ offers a contract $c_i$, which represents her commitment of producing $c_i$ units of energy for a given time period in future. Consider an aggregate with $n$ suppliers offering contracts. The market model contains the following two steps:

\begin{itemize}
\item \textit{ex-ante:} The price of unit of energy is denoted by $f(\hat{c}_{tot})$, where $f: \mathbb{R}^+ \rightarrow \mathbb{R}^+$ and $\hat{c}_{tot}$ is the total contract submitted to system operator. The price is announced by the system operator for a specific trading period. In this step, all the suppliers offer their contracts to the aggregate manager. The contract profile $\cc = [c_1,\hdots,c_n]$ contains all the individual contracts. The feasible set for the contract of supplier $i$ is $[0, \; c_{i,max}]$, where $c_{i,max}$ indicates the nameplate capacity of supplier $i$. Clearly the sets of contracts for the suppliers are compact and convex. The feasible set for supplier $i$ is denoted by $\mathcal{C}_i$ and the set of all possible contract profiles is represented by $\C = \mathcal{C}_1\times \mathcal{C}_2 \times \hdots \times \mathcal{C}_n$. Clearly, the set $\mathcal{C}$ is also a compact and convex set. In addition, we assume that the aggregate contract is the sum of individual contracts: $c = \sum_{i=1}^n c_i$ \cite{BitarOptimal2010,Nayyar2013}. Furthermore, assume that all the suppliers within the aggregate have the same a priori information about the imbalance prices, therefore, identical estimations. Our analysis is based on these estimated imbalance prices.
\item \textit{ex-post:} This step occurs after the trading period, where each supplier realizes an amount of production, e.g., for supplier $i$ the production amount is $w_i \in \mathcal{C}_i$. The supply profile is denoted by $\w = [w_1,\hdots,w_n]$ and the aggregate supply is $w = \sum_{i=1}^n w_i$. Consider $F: \mathcal{C} \rightarrow [0,1]$ as the joint probability distribution of the supply profile on $w$ and $F_i: \mathcal{C}_i \rightarrow [0,1]$ as its marginal distribution on $w_i$. 
\end{itemize}
\begin{remark}
	The imbalance prices are not announced until the {\em ex-post} step. However, suppliers need to have some a priori information about those prices to be able to choose the right strategy in the forward market. We will assume that $\theta = (q, \lambda)$ is the pair of expected imbalance prices, where $q>0$ is the penalty associated with the shortfall and $\lambda$ is the penalty ($\lambda<0$) or bonus ($\lambda \geq 0$) for surplus in production. A potential future path is to extend the analysis proposed in this paper to consider the effect of uncertainty in the estimates of the imbalance prices, when all suppliers have the same or different a priori information.
\end{remark}
\begin{remark}
	Throughout the paper, we refer to negative cost as {\em bonus}.
\end{remark}
Let us denote the deviation from the contract for each supplier with $d_i = c_i-w_i$ and the deviation profile with $\dd = [d_1, d_2 , \hdots , d_n]$. The deviation of the aggregate from its contract, then is represented by $d = \sum_{i=1}^n d_i$. 
The expected payoff of the aggregate from the system is obtained from:
\begin{equation}
\pi(c) = f(\hat{c}_{tot})c -\mathbb{E}[S(d,\theta)] ,
\end{equation}
where $S(d,\theta)$ represents the cost to the aggregate charged by the system operator for $d$ units of energy deviation from the contract, based on the set of expected imbalance prices $\theta$. Since we assumed that $\theta$ is deterministic, the expectation is with respect to $d$, which depends on $w$ that is not known in the {\em ex-ante} market. 
\subsection{System Cost Function}
In the present work, we assume the following cost function for the system, which is also leveraged elsewhere, e.g. \cite{Lin2014}:
\begin{equation} \label{costfunction}
S(d,\theta) = q[d]^+ -\lambda [-d]^+.
\end{equation}
There are other cost function models utilized in the literature \cite{Nayyar2013, BitarBringing2012, Harirchi2014}. The generalization of the results of this paper can be investigated on other types of cost functions, which is subject to future work.


\section{Proportional Cost Sharing Mechanism} \label{Sec:proportional}
The aggregate cost function is used by the aggregate to share the realized cost/bonus, $S(d,\theta)$, among the suppliers based on their individual deviations from the submitted contract. The aggregate cost function is not necessarily the same as the system cost function. In this work, we consider a class of aggregate cost functions that are known as proportional cost functions, and we define them as follows:
\begin{definition} \label{Def:proportional}
	The {\em proportional cost sharing mechanism} is a mapping $\phi$ from $\mathbb{R}^n\times \mathbb{R}^2\times \mathcal{J}$ corresponding to deviation profile, $\dd$, the pair of imbalance prices, $\theta$, and the set of admissible cost functions, $J$, into $\mathbb{R}^n$, where the cost share for supplier $i$, $\phi_i(\dd,\theta,J)$, is defined as:
	\begin{equation} \label{eqn:costfunc}
	\phi_i(\dd,\theta,J) = \frac{J(d,\theta)}{\sum_{j = 1}^{n}J(d_j,\theta)}J(d_i,\theta).
	\end{equation}
An admissible cost function, $J \in \mathcal{J}$, results in a proportional cost sharing mechanism $\phi_i(\dd,\theta,J)$ that satisfies the following properties:
		\begin{enumerate}
			\item Budget balanced: $\sum_{i=1}^n\phi_i(\dd,\theta,J)=S(d,\theta)$
			\item Ex-post individual rationality: If $S(d_i,\theta)\geq 0$, then $\phi_i(\dd,\theta,J) \leq S(d_i,\theta)$
			\item No exploitation: $\phi_i(\dd,\theta,J) = 0$, when $d_i = 0$. 
			\item Fairness: if $d_i = d_j$, then  $\phi_i(\dd,\theta,J) = \phi_j(\dd,\theta,J)$.
		\end{enumerate}
\end{definition}

\textbf{Notation}: Note that $J$ is a function of the deviation profile, the net deviation of aggregate and the imbalance prices. For the sake of simplicity in the notation, we denote it with two arguments as $J(d_i,\theta)$.

The properties mentioned above are used as desired properties of a cost sharing mechanism in the literature \cite{Lin2014,Harirchi2014,BitarBringing2012,Nayyar2013}. There is one more property that is introduced in \cite{Lin2014} for the case with no bonus, $\lambda<0$, which is called monotonicity. We define the monotonicity for the general choice of $\lambda$, where it can be negative or positive.
\begin{definition}
Let $\Delta^+ = \{ i \in \mathcal{N} | d_i \leq 0 \}$ represent the indices for suppliers with surplus, and $\Delta^- = \{ i \in \mathcal{N} | d_i>0 \}$ the indices for suppliers with shortfall. Consider the following two cases:
\begin{itemize}
	\item $\lambda\leq0$: $\phi_i(\dd,\theta,J)$ satisfies monotonicity, if for any $i,j \in \Delta^+$ or $i,j \in \Delta^-$ such that $|d_i|\geq |d_j|$, we have $\phi_i(\dd, \theta,J)  \geq \phi_j(\dd, \theta,J)$.
	\item $\lambda > 0$: $\phi_i(\dd,\theta,J)$ satisfies monotonicity, if for any $i,j \in \Delta^+$ such that $d_i\leq d_j$ or $i,j \in \Delta^-$ such that $|d_i| \geq |d_j|$, we have $\phi_i(\dd, \theta,J)  \geq \phi_j(\dd, \theta,J)$.
\end{itemize}
\end{definition}
Monotonicity can be interpreted as a fairness condition such that if a positive or negative deviation is penalized, the supplier with more deviation has to be penalized no less than a supplier with less deviation, and if surplus is appreciated by bonus, the bonus for a supplier with more surplus must not be less than the bonus for one with less surplus. Clearly, monotonicity is a desired property.

\subsection{Choice of Cost Function}
In this section, we consider two candidates for cost function $J$, and study the advantages and drawbacks of each one. 

\subsubsection{Candidate 1} The first candidate that comes into mind is the system cost function, $S(d,\theta)$. We indicate this function by $\tilde{J}$, and define it as follows:
\begin{equation}\label{eqn:firstcost}
\tilde{J}(d_i,\theta) = q[d_i]^+ - \lambda [-d_i]^+.
\end{equation}
The proportional cost sharing mechanism with this choice of cost function has the following form:
\begin{equation}\label{eqn:propcost1}
\phi_i(\dd,\theta,\tilde{J}) = \alpha(\dd)(q[d_i]^+ - \lambda [-d_i]^+),
\end{equation}
where
\begin{equation} \label{eqn:beta}
\alpha(\dd) = \frac{q[d]^+ -\lambda [-d]^+}{q\sum_{j=1}^n [d_j]^+ -\lambda \sum_{j=1}^n [-d_j]^+}.
\end{equation}

\begin{proposition} \label{pro:prop1}
	According to Def. \ref{Def:proportional}, the cost sharing mechanism $\phi_i(\dd,\theta,\tilde{J})$ is a proportional cost sharing mechanism, if $\lambda \leq 0$.
\end{proposition} 
\begin{proof}
			Let us check each property separately:
	\begin{itemize}
		\item Budget balanced: Clearly, 
		\begin{align*}
		\sum_{i=1}^n \phi_i(\dd,\theta,\tilde{J}) & = \frac{\tilde{J}(d,\theta)}{\sum_{j = 1}^{n}\tilde{J}(d_j,\theta)}\sum_{i=1}^n\tilde{J}(d_i,\theta)\\
		& = \tilde{J}(d,\theta) = S(d,\theta).
		\end{align*}		 
		\item Ex-post individual rationality: By replacing $\tilde{J}(d_i,\theta)$ with $S(d_i,\theta)$, we have:
		\begin{align*}
		\phi_i(\dd,\theta,\tilde{J}) = \frac{q[d]^+ -\lambda [-d]^+}{q\sum_{i=1}^n [d_i]^+ -\lambda \sum_{i=1}^n [-d_i]^+}S(d_i,\theta).
		\end{align*}
		We also know that the following two inequalities hold:
		\begin{align*}
		& \sum_{i=1}^n [d_i]^+ \geq \Big [\sum_{i=1}^n d_i \Big ]^+ =[d]^+\\
		& \sum_{i=1}^n [-d_i]^+ \geq \Big[\sum_{i=1}^n -d_i \Big]^+ =[-d]^+.
		\end{align*}
		Hence, if $\lambda<0$, then
		\begin{align*}
		& \frac{q[d]^+ -\lambda [-d]^+}{q\sum_{i=1}^n [d_i]^+ -\lambda \sum_{i=1}^n [-d_i]^+} \leq 1,\\ & \implies \phi_i(\dd,\theta,\tilde{J}) \leq S(d_i,\theta) .
		\end{align*}
		\item No exploitation: This is clearly satisfied because if $d_i = 0$, then $\tilde{J}(d_i,\theta) = 0$, hence $\phi_i(\dd,\theta,\tilde{J}) = 0$. Note that for the case that all the deviations are zero, the denominator becomes zero. For this particular case, we define $\phi_i(\dd,\theta,\tilde{J})$ separately to be zero.
		\item Fairness: Based on the form of the cost function, $\tilde{J}$, we know that, if $d_i = d_j$, then $\tilde{J}(d_i,\theta) = \tilde{J}(d_j,\theta)$. Therefore:
		\begin{align*}
		& \frac{\tilde{J}(d,\theta)}{\sum_{j = 1}^{n}\tilde{J}(d_j,\theta)}\tilde{J}(d_i,\theta) = \frac{\tilde{J}(d,\theta)}{\sum_{i = 1}^{n}\tilde{J}(d_i,\theta)}\tilde{J}(d_j,\theta).
		\end{align*}
		The equality above is obtained by noticing that the two fractions of both sides are identical, and $\tilde{J}(d_i,\theta) = \tilde{J}(d_j,\theta)$. Then, we have
		\[
		\phi_i(\dd,\theta,\tilde{J}) = \phi_j(\dd,\theta,\tilde{J}).
		\]
	\end{itemize}
\end{proof}
Clearly, if $\lambda \leq 0$, then $\phi(\dd,\theta,\tilde{J})$ satisfies the monotonicity property as well. In other words, if there is no bonus in the market, it can be claimed that $\tilde{J}$ is a decent choice for the cost function, because it satisfies all of the desired properties mentioned in Def. \ref{Def:proportional} as well as monotonicity. However, if we consider the effect of bonus, one can immediately see that $\tilde{J}$ is not even a valid proportional cost sharing mechanism according to Def. \ref{Def:proportional}, because it does not necessarily satisfy ex-post individual rationality property. We illustrate this with the following example.
\begin{example}
	Assume $\lambda >0$, and the deviations of the suppliers from their contracts are such that the following holds:
	\begin{align*}
	q\sum_{i=1}^n [d_i]^+ -\lambda \sum_{i=1}^n [-d_i]^+ = \epsilon,
	\end{align*} 
	where $\epsilon > 0$. Considering the fact that $\lambda >0$, one can see that $\epsilon$ can be very small. On the other hand, assume that the total deviation of the aggregate is $\tilde{d} > 0$, then $S(d,\theta) = q\tilde{d}$. Now, as we can make $\epsilon$ arbitrarily small, let $\epsilon = \frac{q\tilde{d}}{2}$. The share of supplier $i$ is obtained as follows:
	\begin{align*}
	\phi_i(\dd,\theta,\tilde{J}) = \frac{q\tilde{d}}{\frac{q\tilde{d}}{2}}\tilde{J}(d_i,\theta) = 2S(d_i,\theta) > S(d_i, \theta) \text{ if } d_i > 0,
	\end{align*}
	which contradicts the ex-post individual rationality property. Similar analysis can be performed to show the dissatisfaction of the monotonicity property, when bonus is considered.
\end{example}
The above-mentioned problems motivate us to investigate other alternatives that can satisfy all the desired properties for proportional cost sharing mechanisms. 

\subsubsection{Candidate 2} Here, we propose the second candidate cost function:
\begin{equation}
J^*(d_i,\theta) = \mathcal{I}_{(d\geq 0)}q[d_i]^+ - \mathcal{I}_{(d<0)}\lambda [-d_i]^+,
\end{equation}
where $\mathcal{I}$ represents the indicator function that is 1 when the condition in subscript is satisfied and zero otherwise. Intuitively, the cost function introduced above penalizes supplier $i$ only if she increases the cost of the aggregate. In other words, if a supplier has a shortfall in energy, she is only penalized if the aggregate also incurs shortfall. On the other hand, if there is bonus for surplus, then the bonus is shared only amongst the suppliers with surplus.

The proportional cost sharing mechanism with cost function $J^*$ has the following form.
\begin{equation} \label{eqn:secondcost}
\phi_i(\dd,\theta,J^*) = \begin{cases}
\beta^+(\dd)(qd_i) & \text{   if  } d_i \geq 0 \: \& \; d \geq 0 \\
\beta^-(\dd)(\lambda d_i) &  \text{   if  } d_i < 0 \; \& \; d<0 \\
0 & \text{otherwise}
\end{cases},
\end{equation}
where
\begin{align}\label{eqn:eta}
\beta^+(\dd) &= \frac{d}{\sum_{i=1}^n [d_i]^+}, \; \beta^-(\dd) &= \frac{-d}{\sum_{i=1}^n [-d_i]^+}.
\end{align}

\begin{proposition} \label{pro:prop2}
	The cost sharing mechanism $\phi_i(\dd,\theta,J^*)$ is a  proportional cost sharing mechanism according to Def.~\ref{Def:proportional} regardless of the choice of $\lambda$. In addition, it satisfies monotonicity.
\end{proposition}
\begin{proof}
	Let us check all the properties separately.\\
\textit{Budget balanced}: similar to the proof of budget balanced in Proposition \ref{pro:prop1}:
		{\small \begin{align*}
		\sum_{i=1}^n \phi_i(\dd,\theta,J^*) & = \frac{J^*(d,\theta)}{\sum_{j = 1}^{n}J^*(d_j,\theta)}\sum_{i=1}^n J^*(d_i,\theta) = J^*(d,\theta)\\
		& = S(d,\theta).
		\end{align*}}	
\textit{Ex-post individual rationality}: First, consider the following fact:
		{\small\begin{align*}
		S(d_i,\theta) = \begin{cases}
		qd_i & \text{   if  } d_i \geq 0 \\
		\lambda d_i &  \text{   if  } d_i < 0
		\end{cases}.
		\end{align*}}
		Then we can write the cost sharing mechanism as follows:
		\begin{align*}
		\phi_i(\dd,\theta,J^*) = \beta(\dd)S(d_i,\theta) 
		\end{align*}
		where $\beta(\dd)$ is defined as:
		{\small\[
		\beta(\dd) = \begin{cases}
		\beta^+(\dd) & \text{   if  } d_i \geq 0 \: \& \; d \geq 0 \\
		\beta^-(\dd) &  \text{   if  } d_i < 0 \; \& \; d<0 \\
		0 & \text{otherwise}
		\end{cases}.
		\]
		Moreover, we know that the following two inequalities hold:
		{\small \begin{align*}
		& \sum_{i=1}^n [d_i]^+ \geq \Big [\sum_{i=1}^n d_i \Big ]^+ =[d]^+\\
		& \sum_{i=1}^n [-d_i]^+ \geq \Big[\sum_{i=1}^n -d_i \Big]^+ =[-d]^+,
		\end{align*}}}
		which means $\beta(\dd) \leq 1$ for all the cases. Now consider all the possible cases:
		\begin{enumerate}
			\item $d_i \geq 0 \; \& \; d \geq 0$: In this case, clearly $\phi_i(\dd,\theta,J^*)\leq S(d_i,\theta)$.
			\item $d_i < 0 \; \& \; d < 0$: Again as $\beta(\dd) \leq 1$, $\phi_i(\dd,\theta,J^*)\leq S(d_i,\theta)$.
			\item $d_i > 0 \; \& \; d < 0$: In this case, $\phi_i(\dd,\theta,J^*) = 0 \leq S(d_i,\theta)$.
			\item $d_i < 0 \; \& \; d > 0$: In this case we should consider two cases: a) $\lambda > 0$: then $S(d_i,\theta)<0$, b) $\lambda < 0$, which concludes that $\phi_i(\dd,\theta,J^*) = 0 \leq S(d_i,\theta)$.
		\end{enumerate}
		As we can see for all possible cases the ex-post individual rationality is satisfied.
		
\textit{No exploitation}: This is clearly satisfied because $J^*(d_i,\theta) = 0$ if $d_i = 0$.

\textit{Fairness}: Note that $d$ is the same for both suppliers $i$ and $j$, hence $J^*(d_i,\theta) = J^*(d_j,\theta)$. The rest of the proof is similar to the proof of fairness in Proposition \ref{pro:prop1}.

\textit{Monotonicity}: Consider the following cases:
\begin{enumerate}
\item $i,j \in \Delta^-$: Then for $k \in \{i,j\}$, we have $\phi_k(\dd,\theta,J^*) = \beta(\dd)q d_k$, then if $d_i \geq d_j$, then $ \beta(\dd)q d_i \geq  \beta(\dd)q d_j$. Note that $\beta(\dd)$ is same for both suppliers.
\item $i,j \in \Delta^+$: Then for $k \in \{i,j\}$, we have $\phi_k(\dd,\theta,J^*) = \beta(\dd)\lambda d_k$, assume $|d_i| \geq |d_j|$, then if $\lambda \leq 0$, then $\beta(\dd)\lambda d_i \geq \beta(\dd)\lambda d_j$, and if $\lambda > 0$, then $\beta(\dd)\lambda d_j \geq \beta(\dd)\lambda d_i$.
\end{enumerate}
\end{proof}

Next, we turn our attention to the contract game composed of the proposed cost sharing mechanism, and the market properties associated with this contract game.

\section{Contract Game and Existence of Nash Equilibria} \label{Sec:Nash}

\subsection{Contract Game Formulation}
The contract game $G(\mathcal{N},\mathcal{C},\pi)$ inside an aggregate is defined by the following elements:
\begin{itemize}
	\item \textit{Players}: set of all suppliers: $\mathcal{N} = \{1, 2, \hdots ,n \}$.
	\item \textit{Strategies}: set of feasible strategies, $\mathcal{C}$, which is assumed to be convex and compact.
	\item \textit{Payoffs}: payoff for supplier $i$ is defined as:
	\[
	\pi_i := f(\hat{c}_{tot})c_i - \mathbb{E}[\phi_i(\dd,\theta,J)]
	\]
\end{itemize}
and $\pmb{\pi}=[\pi_1, \hdots \pi_n]$ is the payment profile.

We assume perfect competition in the market, which means that $f(\hat{c}_{tot})$ is independent of $c$. In other words, the total supply of the aggregate is negligible compared to the whole grid. In practice, the market clearing price for a trading period, is known before the dispatch process.

\begin{assumption}\label{a:fixp}
	 For the remainder of paper, we explicitly assume: $f(\hat{c}_{tot}) = p$, where $p$ is a constant. In order to avoid trivial market outcome, we also assume: $|\lambda| < p < q$.
\end{assumption}


\subsection{Existence of Nash equilibria}
In the game theory and mechanism design, one of the desired properties that a game leveraging the designed mechanism is required to satisfy is the existence of at least one pure strategy Nash equilibrium. A pure strategy defines a specific move or action that a player will follow in every possible attainable situation in a game. Such moves may not be random, or drawn from a distribution. A pure strategy Nash equilibrium is defined as follows:
\begin{definition}
	A set of pure strategies is called a {\em pure strategy Nash equilibrium} if no player can increase their payoffs by unilaterally changing their strategies \cite{nisan2007algorithmic}.
\end{definition}
With a slight abuse of notation, we call the market with a pure strategy Nash equilibrium, a stable market. In what follows, we check whether or not the cost sharing mechanism defined by $\phi_i(\dd,\theta,J^*)$ constitutes a stable market. The answer to this question is provided by Theorem \ref{thm:J2}, which is the main result of the paper.

\begin{theorem}\label{thm:J2}
	Under Assumption \ref{a:fixp}, the contract game $G = (\mathcal{N},\mathcal{C},\pmb{\pi})$ defined by the proportional cost sharing mechanism with cost function $J^*$, and the set of convex and compact strategies $\mathcal{C}$ has at least one pure strategy Nash equilibrium.
\end{theorem}
\begin{proof}
	See Appendix.
\end{proof}

\begin{remark}
	It can be shown that the game defined by cost sharing mechanism $\phi(\dd,\theta,\tilde{J})$ does not necessarily have a pure-strategy Nash equilibrium, if $\lambda$ is positive. Additionally, the proportional cost sharing mechanism, $\phi(\dd,\theta,\tilde{J})$ penalizes the suppliers who decreased the cost of the aggregate in certain cases, which is not desired. On the other hand, the contract game described by $\phi(\dd,\theta,J^*)$, constitutes at least one pure strategy Nash equilibrium regardless of the sign of $\lambda$, and does not penalize suppliers whose deviations compensate for the net shortfall of the aggregate.
\end{remark}


\section{Illustrative Examples} \label{Sec:example}
In this section, we provide two examples, which illustrate the results on the existence of a pure-strategy Nash equilibrium for the contract game. In the first example, there is no bonus for overproduction. In this example, our goal is to illustrate that the expected payoff function for supplier $i$ is concave in $c_i$. The second example shows that if there is a bonus in the market, the proportional cost sharing mechanism proposed here also results in a contract game with a pure strategy Nash equilibrium. Additionally, for the case that $\lambda >0$, we illustrate that the expected payoff function of supplier $i$ is quasi-concave in $c_i$. 

Consider an aggregate of two wind suppliers with the following wind power probability distributions:
\begin{equation} \label{eqn:ex}
w_1 =\begin{cases}
1 & \text{  w.p.  } 0.7\\
2 & \text{  w.p.  } 0.3
\end{cases}  
\text{   ,   }w_2 =\begin{cases}
1 & \text{  w.p.  } 0.3\\
2 & \text{  w.p.  } 0.7
\end{cases}. 
\end{equation}
The market clearing price is set to be $p = 0.5$ and the expected penalty price for surplus is $q=1.5$. The set of feasible contracts for the two players are $\mathcal{C}_1=\mathcal{C}_2=[0,2]$. 

In order to verify the results of Theorem \ref{thm:J2}, we calculate the Nash equilibrium of the game with cost sharing mechanism defined by $\phi(\dd,\theta,J^*)$ for two cases: $\lambda>0$ and $\lambda<0$.

\begin{example}
	Consider the aggregate described by \eqref{eqn:ex}. Assume the cost for energy surplus is negative, \textit{i.e.} $\lambda = -0.4$. The payoffs are calculated by using cost function $\pi_i$. Fig.~\ref{fig:ex2}.(a) and \ref{fig:ex2}.(b) illustrate the expected payoffs for supplier 1 and 2 versus the strategies of the two players.
	There exists a unique pure strategy Nash equilibrium for this example game at $(c_1,c_2) = (1,2)$, and the payoffs for the two suppliers at the Nash equilibrium are: $(\pi_1,\pi_2) = (0.416,0.685)$. Fig.~\ref{fig:ex2}.(c) illustrates the first supplier's payoff, when $c_2$ is fixed at $1.5$. Clearly $\pi_1$ is concave in $c_1$ and $\pi_2$ is concave in $c_2$.  
\end{example} \vspace{-0.3cm}
\begin{figure}[h!]  
	\includegraphics[width=3.3in]{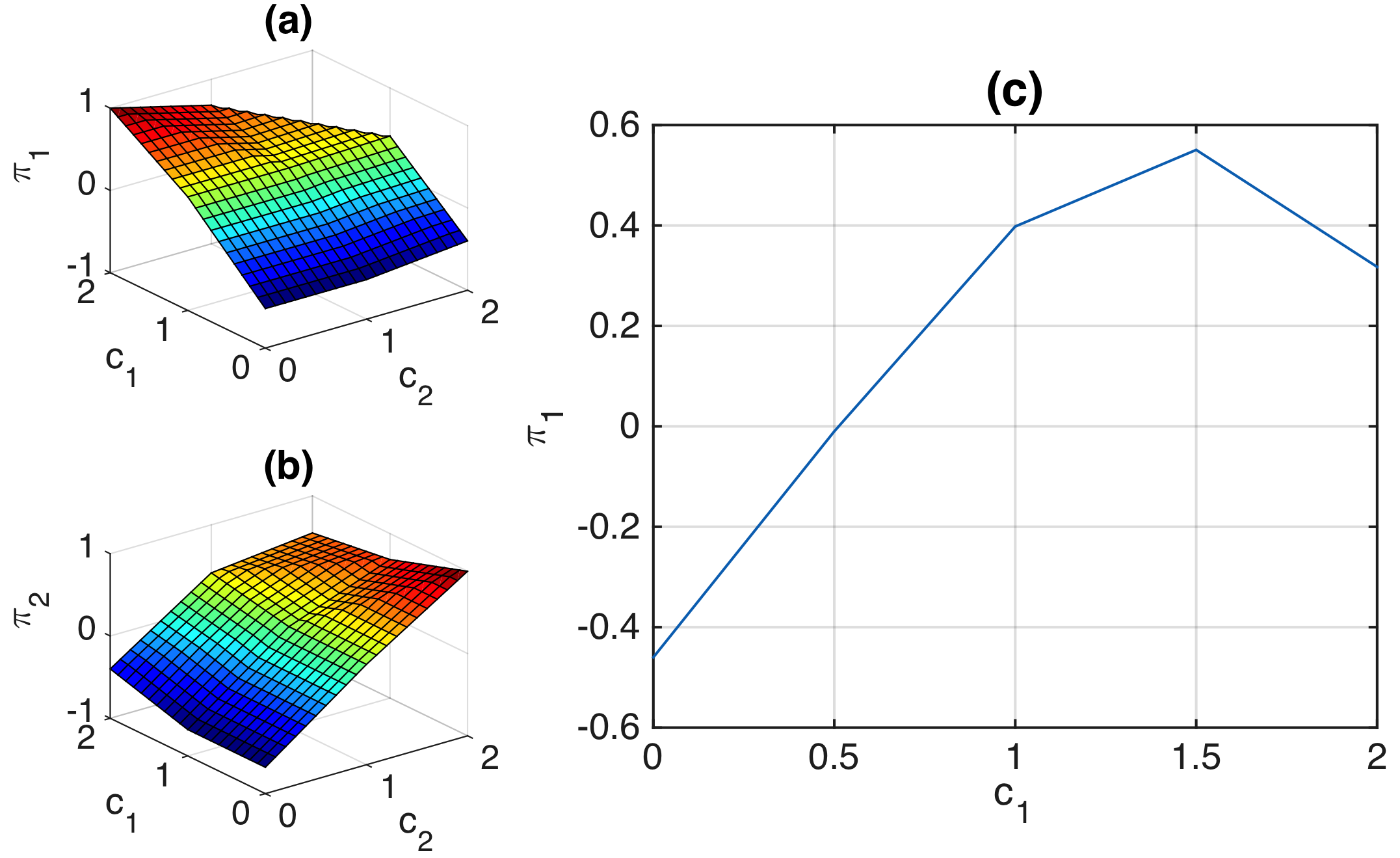}
	\caption{\textbf{(a)}: supplier 1's payoff vs. $c$. $\;$ \textbf{(b)}: supplier 2's payoff vs. $c$. $\quad$
		\textbf{(c)}: illustration of concavity of the payoff of supplier 1 in $c_1$ at $c_2 = 1.5$.}
	\label{fig:ex2}
\end{figure} \vspace{-0.5cm}
\begin{figure}[h!]  
	\includegraphics[width=3.3in]{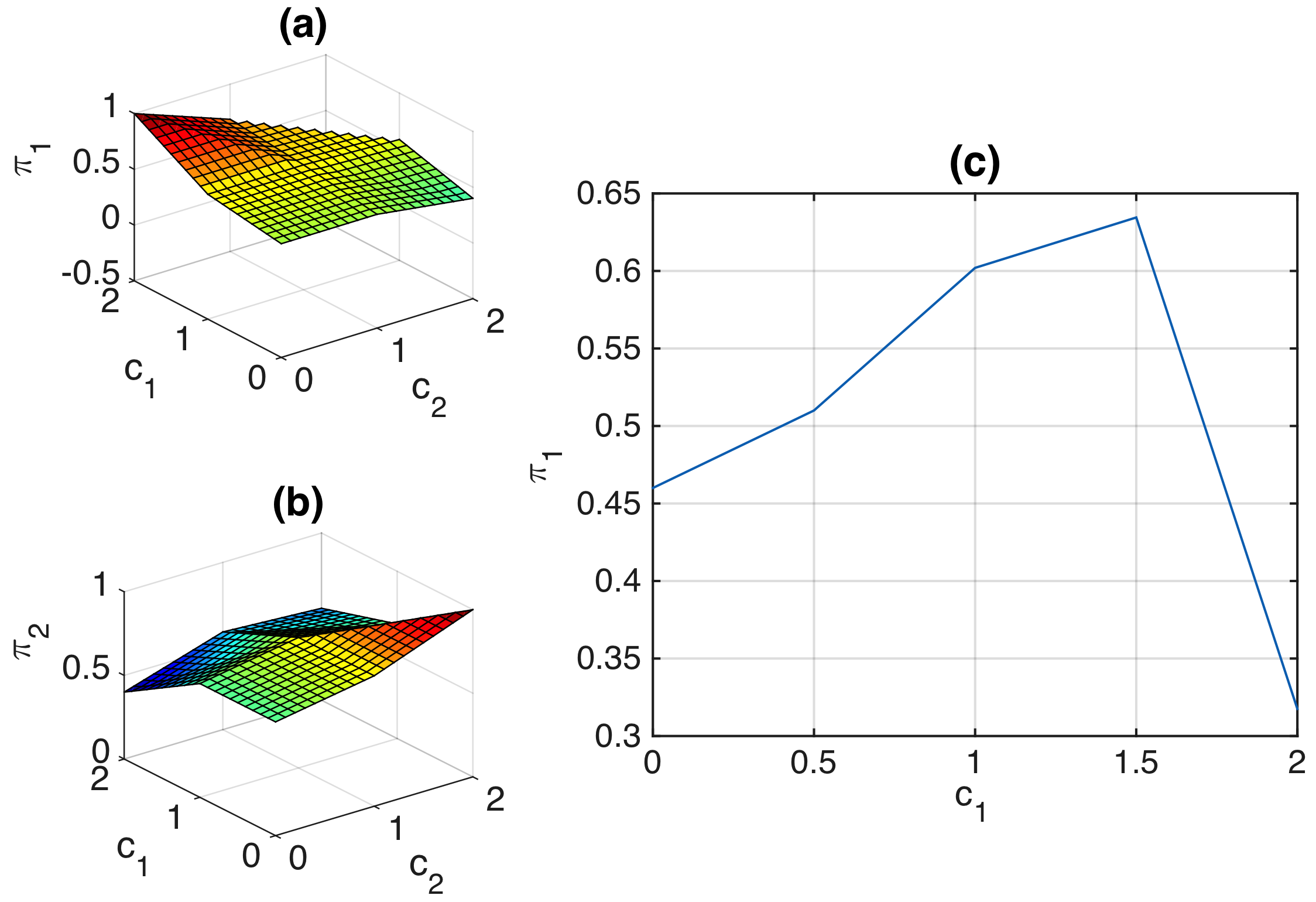}
	\caption{\textbf{(a)}: supplier 1's payoff vs. $c$. $\;$ \textbf{(b)}: supplier 2's payoff vs. $c$. $\quad\quad$ \textbf{(c)}: illustration of quasi-concavity of the payoff of supplier 1 in $c_1$ at $c_2 = 1.5$.}
	\label{fig:ex3}
\end{figure}
\begin{example}
	Consider the same example, but change the value of surplus imbalance price to $\lambda = 0.4$. The payoffs for supplier 1 and 2 versus the strategies of the two players are shown in Fig.~\ref{fig:ex3}.(a) and  \ref{fig:ex3}.(b). The Nash equilibrium for this game occurs at $(c_1,c_2) = (2,1)$ and the pair of payoff values for the two players are $(\pi_1,\pi_2) = (0.685,0.584)$. One can see that the payoff functions $\pi_i$ are quasi-concave in $c_i$ and continuous in both $c_1$ and $c_2$. If we use $\lambda = 0.25$ in this example, there will be infinitely many Nash equilibria on a line segment with equation $c_2 = 3- c_1, c_1,c_2 \in [1 \; 2]$.
\end{example}


\section{Conclusion} \label{Sec:conclusion}
Aggregation of Renewable energy resources is proved to be effective for reducing the uncertainty associated with renewables. In this paper, we used a generalized model of energy markets where surplus of energy is allowed, and the effect of bonus for surplus is considered. Then a proportional cost sharing mechanism is proposed for the generalized market model. This cost sharing mechanism entices the renewable suppliers to join the aggregate, by increasing their payoff compared to when they bid directly to the system operator. It also satisfies the desired properties such as budget balancedness, fairness and monotonicity. Additionally, the net cost of the aggregate is shared among the suppliers proportional to their part in the total cost. In addition, it is proved that the contract game, which leverages this cost sharing mechanism results in at least one pure strategy Nash equilibrium.

Generalization of the results of this paper to the markets with different system cost functions, e.g., the cost function that calculates the real-time deviation rather than the average, is an interesting future direction. Furthermore, the quality of obtained Nash equilibrium can be evaluated by exploiting measures of efficiency loss \cite{Koutsoupiasworst2009, Nisan2007 }.
\bibliographystyle{unsrt}
\bibliography{GTISG}

\begin{appendices}
	\section{Proof of Theorem \ref{thm:J2}}
	\begin{proof}
		Let us first present some preliminary results from convex optimization and game theory literature.
		\subsection{\textbf{Convex optimization results}}
		\begin{lemma} \label{lem:IDexchange} \cite{troutman1995variational}
			Let $f(c,w):\mathcal{C}\times \mathcal{C} \rightarrow \R$ be a differentiable function, where $\mathcal{C}$ is a convex and compact set. Then 
			\begin{equation*}
			\begin{aligned}
			F(c) = \int_{\mathcal{C}} f(c,w) dw
			\end{aligned}
			\end{equation*}
			is differentiable and $	\frac{d F(c)}{d c} = \int_{\mathcal{C}} \frac{\partial f(c,w)}{\partial c}dw$.
			
		\end{lemma}
		\begin{lemma}\label{thm:Econvexity}
			If a function $h(x,y)$ is convex in $x$, then the expectation of $h$, $\mathbb{E}[h(x,y)]$, with respect to $y$ is also convex in $x$.
		\end{lemma}
		\begin{lemma} [Section 3.4 in \cite{Boyd2004}]  \label{lem:quasiconvex} 
			Let $f:\mathcal{C}\rightarrow \R$ be a differentiable function on $\mathcal{C}$ , where $\mathcal{C} \subset \R^n$ is convex and compact. Then $f$ is quasi-concave on the interior of $\mathcal{C}$ if and only if $ \forall \hat c_i,c_i\in \mathcal{C}$:
			\[
			f(\hat c_i)-f(c_i)\geq 0 \implies \frac{\partial(f(c))}{\partial c} \bigg |_{c=c_i}(\hat c_i - c_i) \geq 0.
			\]
		\end{lemma} 
		\subsection{\textbf{Game theoretic results}}
		A (quasi-)concave game is a game in which, the payoff function for each player is (quasi-)concave in her strategy. 
		\begin{lemma} (Debreu, Glicksberg, Fan) \cite{Debreau1952, Fudenberg1991game} \label{lem:quasi}
			A game $G(\mathcal{N},\mathcal{C},\pmb{\pi})$ with compact and convex set of strategies, $\mathcal{C}$, admits at least one pure strategy Nash equilibrium, if the payoff functions $\pi_i$ for all $i\in N$ are continuous in $c$ and quasi-concave in $c_i$. \\
		\end{lemma}
		\subsection{\textbf{Main proof}} 
		\textbf{Notation:} We denote every index except $i$ with subscript $-i$. Also, for the sake of simplicity of the proof, we present functions with respect to $c_i$ and $w_i$ instead of deviation $d_i$.
		Consider two cases:\\
		\textbf{1. } $\pmb{\lambda \leq 0}$
		
		By Lemma \ref{lem:quasi} the proof reduces to showing that $\pi_i(\cc)$ is continuous in $\cc$ and concave in $c_i$ for all $i\in \mathcal{N}$.
		
		With Assumption \ref{a:fixp}: $\pi_i(\cc) = pc_i-\mathbb{E}[\phi_i(\dd,\theta,J^*)]$. $pc_i$ is affine in $c_i$ so it suffices to show that $\mathbb{E}[\phi_i(\dd,\theta,J^*)]$ is convex in $c_i$. Define:
		{\small\[
			\alpha = \sum_{j \in \mathcal{N}, j\neq i} (d_j), \; \alpha_n = \sum_{j \in N, j\neq i} (d_j)^+, \; \alpha_p = \sum_{j \in N, j\neq i} (-d_j)^+.
			\]
		}
		Clearly $\alpha = \alpha_n-\alpha_p$. The proportional cost sharing mechanism can be written as follows:
		\[
		\phi_i(\dd,\theta,J^*) = \begin{cases}
		q\frac{(c_i-w_i+\alpha)(c_i-w_i)}{c_i-w_i+\alpha_n} &\text{If  } c\geq w , c_i \geq w_i\\
		\lambda\frac{(c_i-w_i+\alpha)(c_i-w_i)}{c_i-w_i-\alpha_p} & \text{If  } c< w, c_i < w_i\\
		0  &\text{Otherwise}
		\end{cases}.
		\]
		This function is continuous in $\mathcal{C}$ for both $c$ and $w$. The next step is to show that $\phi_i(c,w)$ is convex in $c_i$.
		Consider the following cases:
		
		\textbf{a)} $c\geq w , c_i \geq w_i$: Let
		\begin{align*}
		g_1(\cc,\w) &= \frac{(\alpha+ c_i-w_i)(c_i-w_i)}{c_i-w_i+\alpha_n}.
		\end{align*}
		The first and second derivatives are
		\begin{align*}
		\frac{\partial g_1(c,w)}{\partial c_i}&= \frac{(c_i - w_i)^2 + 2\alpha_n (c_i - w_i)+ \alpha\alpha_n}{(c_i-w_i + \alpha_n)^2},\\
		\frac{\partial^2 g_1(c,w)}{\partial c_i^2}&=\frac{2\alpha_n (c_i - w_i +\alpha_n)(\alpha_n - \alpha)}{(c_i-w_i + \alpha_n)^4}.
		\end{align*}
		Clearly the second derivative is positive, so $qg_1(c,w)$ is convex in $c_i$ for this region.
		
	\textbf{b)} $c< w , c_i < w_i$: Let
		\begin{align*}
		g_2(\cc,\w) =\frac{(c_i-w_i+\alpha)(c_i-w_i)}{c_i-w_i-\alpha_p}.
		\end{align*}
		The first and second derivatives are
		\begin{align*}
		\frac{\partial g_2(c,w)}{\partial c_i}&=\frac{(c_i - w_i)^2 - 2\alpha_p (c_i - w_i)- \alpha\alpha_p}{(c_i-w_i - \alpha_p)^2}\\
		\frac{\partial^2 g_2(c,w)}{\partial c_i^2} &= -\frac{2\alpha_p (w_i - c_i +\alpha_p)(\alpha_p + \alpha)}{(c_i-w_i - \alpha_p)^4}.
		\end{align*}
		Clearly for $c <w,\; c_i<w_i$, $\frac{\partial^2 g_2(c,w)}{\partial c_i^2} \leq 0$, hence for $\lambda \leq 0$, the function $\lambda g_2$ is convex in $c_i$ for this region. For the other scenarios in $\cc, \w$ space, $\phi_i$ is zero. Therefore, $\phi_i$ is a piecewise function, which is continuous and is convex in $c_i$ in its domain, which implies that $\phi(\dd,\theta,J^*)$ is convex in $c_i$. 
		
		By Lemma \ref{thm:Econvexity}, $\mathbb{E}[\phi_i(\dd,\theta,J^*)]$ is also convex in $c_i$. This implies concavity of $pc_i-\mathbb{E}[\phi_i(\dd,\theta,J^*)]$. Therefore, the contract game $G$ is concave, and has at least one pure strategy Nash equilibrium. 
		
		\textbf{2. } $\pmb{\lambda > 0}$
		
		Without loss of generality, we assume the following:
		\begin{itemize}
			\item The distribution function, $F(w)$, is integrable in the domain $\mathcal{C}$.
		\end{itemize} 
		\begin{remark}
			For discrete probability mass function, $\tilde{F}(w)$ the proof is similar, except that the integrals are replaced with sums.
		\end{remark} 
		Define three regions for $c_i$ as:
		{\scriptsize \begin{equation*}
			\begin{aligned}
			\begin{cases}
			\mathcal{R}_1 & \text{If  }c_i\leq \min \{w_i, w_i+\alpha\}\\
			\mathcal{R}_2&\text{If  } \min \{w_i, w_i+\alpha\} \leq c_i\leq \max \{w_i, w_i+\alpha\}\\
			\mathcal{R}_3  & \text{If  } c_i\geq \max \{w_i, w_i+\alpha\} 
			\end{cases}.
			\end{aligned}
			\end{equation*}}
		
		Let us first consider $\hat \pi_i(\dd,\theta,J^*) = \mathbb{E}_w \big[pc_i - \phi_i(\dd,\theta,J^*)\big]$. The objective is to show that this function is quasi-concave. This function is defined as: 
		
		{\scriptsize \begin{equation*}
			\begin{aligned}
			\hat \pi_i(\dd,\theta,J^*)= \begin{cases}
			\int_{\mathcal{W}_1} \big[pc_i-\lambda\frac{(c_i-w_i+\alpha)(c_i-w_i)}{(c_i-w_i-\alpha_p)} \big]F(w)dw & \text{If  }c_i\in \mathcal{R}_1 \\
			\int_{\mathcal{W}_2}pc_iF(w)dw  & \text{If } c_i\in \mathcal{R}_2 \\
			\int_{\mathcal{W}_3}\big[pc_i-q\frac{(c_i-w_i+\alpha)(c_i-w_i)}{c_i-w_i+\alpha_n} \big]F(w)dw &\text{If  } c_i\in \mathcal{R}_3\\
			\end{cases},
			\end{aligned}
			\end{equation*}}
		where $\mathcal{W}_i$ denotes the set of power profiles that cause $c_i$ to be in $\mathcal{R}_i$. If $c_i \in \mathcal{R}_1$, then we show that the function $\hat \pi_i$ is monotonically non-decreasing in $c_i$. We calculate the derivative as:
		
		{\scriptsize \begin{equation*}
			\begin{aligned}
			\int_{\mathcal{W}_1} \Big[p-\lambda \frac{(c_i - w_i)^2 - 2\alpha_p (c_i - w_i)- \alpha\alpha_p}{(c_i-w_i - \alpha_p)^2} \Big]F(w)dw.
			\end{aligned}
			\end{equation*}}
		We used Lemma \ref{lem:IDexchange} to derive the above equation. Inside the integral can be written as the following:
		{\scriptsize \begin{equation*}
			\begin{aligned}
			& \frac{(p-\lambda)(c_i - w_i)^2 + (\lambda-p)\alpha_p (c_i - w_i)}{(c_i-w_i - \alpha_p)^2} \\
			& \quad \quad \quad \quad+ \frac{\lambda \alpha_p (c_i-w_i+\alpha) -p\alpha_p(c_i-w_i-\alpha_p)}{(c_i-w_i - \alpha_p)^2} \geq 0, \text{  in }\mathcal{R}_1,
			\end{aligned}
			\end{equation*}}
		
		 because $p>\lambda$ and $c_i-w_i-\alpha_p \leq c_i-w_i+\alpha \leq 0$. This proves that $\hat \pi_i(\dd,\theta,J^*)$ is monotonically non-decreasing in region $\mathcal{R}_1$.
		
		Now consider $c_i \in \mathcal{R}_2$. In this case, the derivative inside the integral is $p \geq 0$. Hence, $\hat \pi_i(\dd,\theta,J^*)$ is monotonically increasing in region $\mathcal{R}_2$.
		
		Finally, from the proof for $\lambda \leq 0$, we know that $\hat \pi_i(\dd,\theta,J^*)$ is concave in $\mathcal{R}_3$. Observe that $\hat \pi_i(\dd,\theta,J^*)$ is a continuous function in $c_i$. Consider the following two functions:
		{\scriptsize \begin{equation*}
			\begin{aligned}
			t_{i,1}(\dd,\theta,J^*)= \begin{cases}
			\int_{\mathcal{W}_1} \big[pc_i-\lambda\frac{(c_i-w_i+\alpha)(c_i-w_i)}{(c_i-w_i-\alpha_p)} \big]F(w)dw & \text{If  }c_i\in \mathcal{R}_1 \\
			\int_{\mathcal{W}_2}pc_iF(w)dw  & \text{If } c_i\in \mathcal{R}_2 \\
			0 &\text{If  } c_i\in \mathcal{R}_3\\
			\end{cases}.
			\end{aligned}
			\end{equation*}}
		{\scriptsize \begin{equation*}
			\begin{aligned}
			t_{i,2}(\dd,\theta,J^*)= \begin{cases}
			0  & \text{If  }c_i\in \mathcal{R}_1 \\
			0  & \text{If } c_i\in \mathcal{R}_2 \\
			\int_{\mathcal{W}_3}\big[pc_i-q\frac{(c_i-w_i+\alpha)(c_i-w_i)}{c_i-w_i+\alpha_n} \big]F(w)dw &\text{If  } c_i\in \mathcal{R}_3\\
			\end{cases}.
			\end{aligned}
			\end{equation*}}
		It is straightforward to show that both $t_{i,1}$ and $t_{i,2}$ are quasi-concave functions by using Lemma \ref{lem:quasiconvex}. A quasi-concave function increases to a peak and decreases afterwards. The point-wise maximum of quasi-concave functions is also quasi-concave (Chapter 3 in \cite{Boyd2004}). Therefore, $\hat \pi_i = \max \{t_{i,1},t_{i,2}\}$ is a quasi-concave function. Finally, note that $\pi_i(\dd,\theta,J^*) = \hat \pi_i(\dd,\theta,J^*)$, hence it is quasi-concave. By Lemma \ref{lem:quasi}, the contract game $G = (\mathcal{N}, \mathcal{C}, \pi)$ is a quasi-concave game and has at least one pure-strategy Nash equilibrium. 
	\end{proof}	
\end{appendices}


\addtolength{\textheight}{-12cm}  	
	
\end{document}